\documentclass[11pt,a4paper]{article}
\usepackage[utf8]{inputenc}
\usepackage{amsmath,amsthm,amssymb,amsfonts}
\usepackage{geometry}
\usepackage{hyperref}
\usepackage[numbers]{natbib}

\theoremstyle{plain}
\newtheorem{thm}{Theorem}[section]
\newtheorem{lem}[thm]{Lemma}
\newtheorem{prop}[thm]{Proposition}
\newtheorem{cor}[thm]{Corollary}
\newtheorem{conj}[thm]{Conjecture}

\theoremstyle{definition}
\newtheorem{defn}[thm]{Definition}
\newtheorem{rmk}[thm]{Remark}

\newtheorem{ques}[thm]{Question}

\newcommand{\Fq}{\mathbb F_q}
\newcommand{\qbinom}[2]{\genfrac{[}{]}{0pt}{}{#1}{#2}_q}
\newcommand{\points}[1]{\qbinom{#1}{1}}
\newcommand{\Gr}[2]{\mathrm{Gr}_{#2}(#1)}
\newcommand{\cA}{\mathcal A}
\newcommand{\cB}{\mathcal B}
\newcommand{\cC}{\mathcal C}
\newcommand{\cF}{\mathcal F}
\newcommand{\cG}{\mathcal G}

\newcommand{\cM}{\mathcal M}
\newcommand{\cR}{\mathcal R}

\title{Large point-degrees in intersecting families of finite vector spaces}
\author{
  Jicheng Ma\thanks{School of Mathematics, Renmin University of China, Beijing, 100872, China. Email: mjc191812@ruc.edu.cn}
\and
  Hao Li\thanks{School of Mathematics, Renmin University of China, Beijing, 100872, China. Email: hlimath@ruc.edu.cn}
}
\date{}

\begin{document}
\maketitle

\begin{abstract}
Let \(V\) be an \(n\)-dimensional vector space over the finite field \(\Fq\), and let \(\Gr{V}{k}\) denote the family of all \(k\)-dimensional subspaces of \(V\).
A family \(\cF\subseteq\Gr{V}{k}\) is called intersecting if \(\dim(F\cap F')\ge1\) for all \(F,F'\in\cF\).
For a point \(P\le V\), let \(d_P(\cF)\) denote the number of members of \(\cF\) that contain \(P\), and order the point-degrees as \(d_1(\cF)\ge d_2(\cF)\ge\cdots\ge d_{\points{n}}(\cF)\), where \(\points{m}=(q^m-1)/(q-1)\) is the number of points in an \(m\)-dimensional subspace.
Recent work of Frankl and Wang~\cite{FW2025} and of Huang and Rao~\cite{HR2026} established that for \(k\)-uniform intersecting families \(\cF\subseteq\binom{[n]}{k}\) with \(n\ge2k+1\), the bound \(\binom{n-2}{k-2}\) governs the order statistic \(d_{2k+1}(\cF)\).
We prove that every intersecting family \(\cF\subseteq\Gr{V}{k}\) with \(n\ge2k+1\) satisfies \(d_{\points{k}^{2}}(\cF)\le\qbinom{n-2}{k-2}\).
The naive \(q\)-analog of the Huang--Rao \((k+2)\)-th degree theorem fails, as a vector-space Hilton--Milner construction has \(\points{k+1}\) points of degree larger than \(\qbinom{n-2}{k-2}\); we prove the corrected bound \(d_{1+\points{k+1}}(\cF)\le\qbinom{n-2}{k-2}\) for fixed \(q\), sufficiently large \(k\), and \(n>3k\), using a structural theorem of Ihringer and Kupavskii~\cite{IK2026}.
For larger degree indices \(i\), a saturated Frankl--Hilton--Milner family of Ihringer and Kupavskii identifies the conjectural sharp bound on \(d_{\points{i+1}}(\cF)\), and we prove two necessary conditions that any strict counterexample to this conjecture must satisfy.
\end{abstract}




\section{Introduction}
\label{sec:introduction}

We denote by \([n]\) the standard \(n\)-element set \(\{1,2,\dots,n\}\), by \(\binom{[n]}{k}\) the family of all \(k\)-element subsets of \([n]\), and by \(2^{[n]}\) the power set of \([n]\).
For a family \(\mathcal F\subseteq 2^{[n]}\), we say that \(\mathcal F\) is \(t\)-intersecting if \(|A\cap B|\ge t\) for all \(A,B\in\mathcal F\).
If \(\mathcal F\) is \(1\)-intersecting, we simply call it intersecting.
If a \(t\)-intersecting family \(\mathcal F\) is contained in \(\binom{[n]}{k}\), we call it a \(k\)-uniform \(t\)-intersecting family.

The Erd\H{o}s--Ko--Rado theorem~\cite{EKR1961} states that if \(\mathcal F\) is a \(k\)-uniform intersecting family and \(n\ge2k\), then \(|\mathcal F|\le\binom{n-1}{k-1}\); see~\cite{FT2018} for a textbook exposition.
Moreover, equality holds when \(\mathcal F\) consists of all \(k\)-element sets containing a fixed element \(i\), that is, when \(\mathcal F\) is isomorphic to a \(1\)-star.
For this family, the element \(i\) appears in \(\binom{n-1}{k-1}\) members of \(\mathcal F\), while every other element appears in \(\binom{n-2}{k-2}\) members.
For a family \(\mathcal F\subseteq\binom{[n]}{k}\) and \(i\in[n]\), define the degree of \(i\) by
\[
  d_i(\mathcal F):=|\{A\in\mathcal F:i\in A\}|.
\]
When the family \(\mathcal F\) is clear from context, we simply write \(d_i\).
Without loss of generality, we order the degrees as \(d_1\ge d_2\ge d_3\ge\cdots\ge d_n\).

The study of degree conditions in intersecting families goes back at least to Frankl~\cite{Frankl1987maxdeg}, who proved sharp upper bounds on the maximum degree under additional restrictions; the structure of extremal \(t\)-intersecting families was subsequently settled by the complete intersection theorem of Ahlswede and Khachatrian~\cite{AK1997}.
In 2017, Huang and Zhao~\cite{HZ2017} used spectral graph theory to prove that in a \(k\)-uniform intersecting set family \(\mathcal F\), the minimum degree \(d_n(\mathcal F)\) is at most \(\binom{n-2}{k-2}\).
Frankl and Wang~\cite{FW2025} extended this by showing that the upper bound \(\binom{n-2}{k-2}\) governs not only the minimum degree but most of the degrees.
They proved that for \(n>6k-9\), the inequality \(d_{2k+1}(\mathcal F)\le\binom{n-2}{k-2}\) holds, and conjectured that the same bound should remain valid for all \(n>2k\).
Huang and Rao~\cite{HR2026} verified this conjecture and proved further large-degree results.

\begin{thm}[Huang--Rao, Theorem~1.2 of~\cite{HR2026}]\label{thm:2kplus1-degree}
Let \(n\ge2k+1\), and let \(\mathcal F\subseteq\binom{[n]}{k}\) be an intersecting family.
Then
\[
  d_{2k+1}(\mathcal F)\le \binom{n-2}{k-2}.
\]
\end{thm}

\begin{thm}[Huang--Rao, Theorem~1.4 of~\cite{HR2026}]\label{thm:kplus2-degree}
For all sufficiently large \(k\), if \(n>12k\) and \(\mathcal F\subseteq\binom{[n]}{k}\) is intersecting, then
\[
  d_{k+2}(\mathcal F)\le \binom{n-2}{k-2}.
\]
\end{thm}

They also proved a large-\(\ell\) degree theorem~\cite[Theorem 1.6]{HR2026}.

The problems in extremal set theory have natural extensions to families of subspaces over a finite field; we refer to Frankl and Wilson~\cite{FW1986} for the foundational vector-space \(t\)-intersection theory. Let \(q\) be a prime power, and let \(V\) be an \(n\)-dimensional vector space over \(\Fq\).
We write \(\Gr{V}{k}\) for the family of all \(k\)-dimensional subspaces of \(V\), and
\[
  \qbinom{n}{k}
  =
  \prod_{j=0}^{k-1}\frac{q^{n-j}-1}{q^{k-j}-1}
\]
for the Gaussian binomial coefficient.
We use the convention that \(\qbinom{a}{b}=0\) when \(b<0\) or \(b>a\).
For brevity, write
\[
  \points{m}:=\qbinom{m}{1}=\frac{q^m-1}{q-1}
\]
for the number of points in an \(m\)-dimensional vector space over \(\Fq\).
Throughout the paper, a line means a \(2\)-dimensional subspace.
A family \(\cF\subseteq\Gr{V}{k}\) is called intersecting if
\[
  \dim(F\cap F')\ge1
  \qquad\text{for all }F,F'\in\cF.
\]
For a point, that is, a \(1\)-dimensional subspace \(P\le V\), define its point-degree in \(\cF\) by
\[
  d_P(\cF)=|\{F\in\cF:P\le F\}|.
\]
Order all point-degrees as
\[
  d_1(\cF)\ge d_2(\cF)\ge\cdots\ge d_{\points n}(\cF).
\]

Several finite-vector-space degree and stability results are closely related to this setting.
The vector-space Hilton--Milner theorem of Blokhuis, Brouwer, Chowdhury, Frankl, Mussche, Patk\'os and Sz\H{o}nyi~\cite{Blokhuis2010}, together with the binary boundary case supplied by Wang, Xu and Zhang~\cite{WXZ2023}, is the main input in our first theorem.
Shan and Zhou~\cite{SZ2024} proved a \(d\)-degree Erd\H{o}s--Ko--Rado theorem for finite vector spaces, which concerns minimum degrees over higher-dimensional subspaces.
Ihringer and Kupavskii~\cite{IK2026} developed a structural theory for large \(t\)-intersecting families of subspaces and proved a vector-space Frankl degree-type theorem; their structure theorem is the main input for our corrected \(k+2\) result, while their Frankl degree-type examples motivate the large-\(i\) conjecture below.
Their own Frankl degree-type theorem bounds the size of a non-star portion of an intersecting family; it does not directly bound the order statistic \(d_{1+\points{k+1}}(\cF)\) considered in our Theorem~\ref{thm:q-kplus2-cq3}, which is the natural \(q\)-analog of the Huang--Rao \((k+2)\)-th degree theorem.

The natural \(q\)-analog of the index \(2k+1\) is \(1+\points{2k}\), which specializes to \(2k+1\) in the formal limit \(q\to1\).
We prove a stronger statement at the smaller index \(\points{k}^{2}\).

\begin{thm}[A stronger \(q\)-Frankl--Wang theorem]\label{thm:q-fw}
Let \(k\ge2\), let \(V\) be \(n\)-dimensional over \(\Fq\), let \(n\ge2k+1\), and let \(\cF\subseteq\Gr{V}{k}\) be intersecting.
Then
\[
  d_{\points{k}^{2}}(\cF)
  \le
  \qbinom{n-2}{k-2}.
\]
Consequently,
\[
  d_{1+\points{2k}}(\cF)
  \le
  \qbinom{n-2}{k-2}.
\]
\end{thm}

\begin{rmk}\label{rmk:q-scale-strengthening}
For every \(q\ge2\), one has \(\points{k}^{2}<\points{2k}\), so Theorem~\ref{thm:q-fw} is strictly stronger than the statement at the index \(1+\points{2k}\).
In the formal limit \(q\to1\), the comparison reverses to \(k^{2}>2k\) for \(k\ge3\); the consequence at the index \(1+\points{2k}\) is the direct \(q\to1\) analog of the Huang--Rao theorem.
\end{rmk}

Already the vector-space Hilton--Milner family~\cite{HiltonMilner1967,Blokhuis2010}
\[
  \{F\in\Gr{V}{k}:E\le F,\ \dim(F\cap L)\ge1\}
  \cup
  \{L\},
\]
where \(L\) is a \(k\)-space and \(E\nleq L\) is a point, has \(1+\points{k}\) points of degree larger than \(\qbinom{n-2}{k-2}\).
This shows that the literal \(q\)-analogue of the index \(2k+1\) fails for fixed \(q>1\) and large \(k\).

A similar phenomenon affects the set-theoretic \(k+2\) theorem of Huang and Rao.
The vector-space Hilton--Milner family
\[
  \cM(E,L)
  =
  \{F\in\Gr{V}{k}:E\le F,\ \dim(F\cap L)\ge1\}
  \cup
  \{G\in\Gr{E+L}{k}:E\nleq G\},
\]
where \(E\) is a point and \(L\) is a \(k\)-space with \(E\nleq L\), has \(\points{k+1}\) points of degree greater than \(\qbinom{n-2}{k-2}\) (Proposition~\ref{prop:q-kplus2-false}).
The index therefore cannot be improved below \(1+\points{k+1}\) in any range where the upper bound holds.
We formulate this as follows.

\begin{conj}[A corrected \(q\)-analog of the \(k+2\) theorem]\label{conj:q-kplus2-corrected}
There is a constant \(C_q\) such that, for all sufficiently large \(k\) and all \(n>C_q\,k\), every intersecting family \(\cF\subseteq\Gr{V}{k}\) satisfies
\[
  d_{1+\points{k+1}}(\cF)
  \le
  \qbinom{n-2}{k-2}.
\]
\end{conj}

We prove this corrected conjecture in a linear range with an explicit constant.

\begin{thm}[The corrected \(q\)-analog for \(n>3k\)]\label{thm:q-kplus2-cq3}
Fix \(q\).
There exists \(k_0=k_0(q)\) such that, whenever \(k\ge k_0\), \(n>3k\), and \(\cF\subseteq\Gr{V}{k}\) is intersecting, one has
\[
  d_{1+\points{k+1}}(\cF)
  \le
  \qbinom{n-2}{k-2}.
\]
Thus the multiplicative constant \(3\) is admissible in Conjecture~\ref{conj:q-kplus2-corrected} for each fixed \(q\), with the threshold \(k_0\) depending on \(q\).
\end{thm}

\begin{rmk}\label{rmk:thm14-threshold}
The proof of Theorem~\ref{thm:q-kplus2-cq3} invokes the Ihringer--Kupavskii structure theorem (Theorem~\ref{thm:ik-structure} below), which assumes \(n\ge2k+3\).
For \(k\ge3\) this is implied by the linear hypothesis \(n>3k\), so the threshold \(n>3k\) is the only effective hypothesis on the dimension.
The dependence of \(k_0(q)\) on \(q\) is not made explicit; see Lemma~\ref{lem:cq3-estimates}.
\end{rmk}

For larger degree order statistics, the finite-vector-space analog of Huang and Rao's large-\(\ell\) theorem~\cite[Theorem 1.6]{HR2026} again requires a corrected scale.
The relevant examples are the Frankl degree-type families appearing in the work of Ihringer and Kupavskii~\cite[Theorem 4.5 and Section 4.4]{IK2026}.
Given a point \(X\) and an \(i\)-space \(Y_i\) with \(X\cap Y_i=0\), the saturated Frankl--Hilton--Milner family is
\[
  \cG_i
  =
  \{F\in\Gr{V}{k}:X\le F,\ \dim(F\cap Y_i)\ge1\}
  \cup
  \{G\in\Gr{V}{k}:Y_i\le G+X,\ G\cap X=0\}.
\]
Section~\ref{sec:constructions} verifies that \(\cG_i\) is intersecting and that, for every point \(P\le X+Y_i\) with \(P\ne X\),
\[
  d_P(\cG_i)
  =
  \qbinom{n-2}{k-2}
  +
  q^{k-1}\qbinom{n-i-1}{k-i}.
\]
Consequently
\[
  d_{\points{i+1}}(\cG_i)
  =
  \qbinom{n-2}{k-2}
  +
  q^{k-1}\qbinom{n-i-1}{k-i}.
\]
This determines the conjectural upper bound.

\begin{conj}[Corrected large-\(i\) degree conjecture]\label{conj:q-large-i-corrected}
Fix \(q\) and \(\varepsilon>0\). There exists a constant \(C_{\varepsilon,q}\) such that the following holds.
Whenever \(k\) is sufficiently large, \(\varepsilon k\le i\le k\), \(n>C_{\varepsilon,q}\,k\), and \(\cF\subseteq\Gr{V}{k}\) is intersecting, one has
\[
  d_{\points{i+1}}(\cF)
  \le
  \qbinom{n-2}{k-2}
  +
  q^{k-1}\qbinom{n-i-1}{k-i}.
\]
\end{conj}

We do not prove Conjecture~\ref{conj:q-large-i-corrected} in full.
Instead, we prove two necessary conditions on any strict counterexample (Propositions~\ref{prop:anti-clustering-reduction} and~\ref{prop:one-member-loss-reduction}): such a family must either have very large diversity outside a maximum point-star, or contain a single non-star member that captures many noncentral high-degree points.

The remainder of the paper is organized as follows.
In Section~\ref{sec:preliminaries} we collect the vector-space extremal results and elementary estimates used in the proofs.
Sections~\ref{sec:proof-qfw} and~\ref{sec:proof-kplus2} prove the \(q\)-Frankl--Wang bound and the corrected \(k+2\) bound, respectively.
Section~\ref{sec:constructions} gives the sharpness constructions, verifies the properties of \(\cG_i\), and proves the necessary conditions on large-\(i\) counterexamples.
We conclude in Section~\ref{sec:conclusion} by isolating the remaining obstruction to Conjecture~\ref{conj:q-large-i-corrected}.

\section{Preliminaries}
\label{sec:preliminaries}

We recall two standard extremal theorems for finite vector spaces.
The first is the vector-space Erd\H{o}s--Ko--Rado theorem of Hsieh~\cite{Hsieh1975}.

\begin{thm}[Hsieh~\cite{Hsieh1975}]\label{thm:vector-ekr}
Let \(n\ge2k+1\), and let \(\cA\subseteq\Gr{V}{k}\) be intersecting.
Then
\[
  |\cA|\le \qbinom{n-1}{k-1}.
\]
Moreover, equality holds if and only if \(\cA\) is a \(1\)-star, namely all \(k\)-subspaces containing a fixed point.
\end{thm}

The next theorem is the vector-space Hilton--Milner theorem.
The result was proved by Blokhuis, Brouwer, Chowdhury, Frankl, Mussche, Patk\'os and Sz\H{o}nyi~\cite{Blokhuis2010} for all \(n\ge2k+1\) except the binary boundary \(q=2,n=2k+1\).
That remaining case follows from the \(q\)-Kneser Kruskal--Katona theorem of Wang, Xu and Zhang~\cite{WXZ2023}.

\begin{thm}[Vector-space Hilton--Milner~\cite{Blokhuis2010,WXZ2023}]\label{thm:q-hm}
Let \(k\ge3\) and \(n\ge2k+1\).
Let \(\cA\subseteq\Gr{V}{k}\) be intersecting and suppose that
\[
  \bigcap_{A\in\cA}A=0.
\]
Then
\[
  |\cA|
  \le
  f(n,k,q),
\]
where
\[
  f(n,k,q)
  =
  \qbinom{n-1}{k-1}
  -
  q^{k(k-1)}\qbinom{n-k-1}{k-1}
  +
  q^k.
\]
\end{thm}

\begin{lem}[Disjoint subspace count]\label{lem:disjoint-subspace-count}
Let \(\dim W=N\), and let \(M\le W\) have dimension \(m\).
Then the number of \(r\)-dimensional subspaces \(R\le W\) satisfying \(R\cap M=0\) is
\[
  q^{mr}\qbinom{N-m}{r}.
\]
\end{lem}

\begin{proof}
Let \(\pi:W\to W/M\) be the quotient map.
If \(R\cap M=0\), then \(\pi(R)\) is an \(r\)-dimensional subspace of \(W/M\).
Conversely, after choosing an \(r\)-subspace \(\overline R\le W/M\), the subspaces \(R\le \pi^{-1}(\overline R)\) that are disjoint from \(M\) and project isomorphically onto \(\overline R\) are the graphs of linear maps \(\overline R\to M\).
There are \(q^{mr}\) such maps for each of the \(\qbinom{N-m}{r}\) choices of \(\overline R\).
\end{proof}

We only use the upper bound in Theorem~\ref{thm:q-hm}; the equality cases are not needed here.
We shall only need the following consequence of the explicit Hilton--Milner bound.

\begin{lem}[Hilton--Milner upper bound]\label{lem:hm-convenient}
Let \(k\ge3\), \(n\ge2k+1\), and put
\[
  T=\qbinom{n-2}{k-2}.
\]
Then
\[
  f(n,k,q)\le \points{k}\,T.
\]
\end{lem}

\begin{proof}
Let \(E\) be a point and let \(U\) be a \(k\)-subspace with \(E\nleq U\).
The standard vector-space Hilton--Milner construction
\[
  \cM(E,U)
  =
  \{F\in\Gr{V}{k}:E\le F,\ \dim(F\cap U)\ge1\}
  \cup
  \{F\in\Gr{E+U}{k}:E\nleq F\}.
\]
This construction has size \(f(n,k,q)\): the second part has size \(q^k\), since inside the \((k+1)\)-space \(E+U\) these are precisely the complements of the point \(E\); the first part has size
\[
  \qbinom{n-1}{k-1}
  -
  q^{k(k-1)}\qbinom{n-k-1}{k-1},
\]
since after quotienting by \(E\), the excluded members of the \(E\)-star are the \((k-1)\)-subspaces of \(V/E\) disjoint from \((E+U)/E\), and Lemma~\ref{lem:disjoint-subspace-count} gives the subtracted term.
It is therefore enough to prove \(|\cM(E,U)|\le\points{k}T\).

Let
\[
  \cA=\{F\in\Gr{V}{k}:E\le F,\ \dim(F\cap U)\ge1\}.
\]
For each point \(P\le U\), set
\[
  \cA_P=\{F\in\Gr{V}{k}:E+P\le F\}.
\]
Then \(|\cA_P|=T\), and \(\cA\subseteq\bigcup_{P\le U}\cA_P\).
Thus
\[
  \sum_{P\le U}|\cA_P|=\points{k}T.
\]
This cover overcounts \(\cA\).
Indeed, for every hyperplane \(S\le U\), the \(k\)-space \(E+S\) belongs to \(\cA\) and is counted once for each point of \(S\), namely \(\points{k-1}\) times.
The subspaces \(E+S\) are distinct as \(S\) varies over the hyperplanes of \(U\).
Each such subspace therefore contributes at least \(\points{k-1}-1\) to the overcount.
There are \(\points{k}\) hyperplanes in \(U\), so the total overcount is at least
\[
  \points{k}\bigl(\points{k-1}-1\bigr)
  =
  \points{k}\,q\,\points{k-2}
  \ge q^k,
\]
where the identity \(\points{k-1}-1=q\points{k-2}\) follows from the Gaussian recurrence, and the last inequality uses the elementary bound \(\points{m}\ge q^{m-1}\), giving
\[
  \points{k}\,q\,\points{k-2}
  \ge
  q^{k-1}\cdot q\cdot q^{k-3}
  =
  q^{2k-3}
  \ge
  q^{k}
\]
for every \(k\ge3\) and every \(q\ge2\).
Consequently
\[
  |\cA|\le \points{k}T-q^k.
\]
Adding the \(q^k\) members of \(\Gr{E+U}{k}\) not containing \(E\), we obtain
\[
  f(n,k,q)=|\cM(E,U)|\le \points{k}T.
\]
\end{proof}

We also need a simple classification of pairwise intersecting line families.

\begin{lem}[Intersecting line families]\label{lem:line-classification}
Let \(\cF\subseteq\Gr{V}{2}\) be a family of pairwise intersecting lines.
Then either all lines in \(\cF\) contain a common point, or all lines in \(\cF\) lie in a fixed \(3\)-dimensional subspace.
\end{lem}

\begin{proof}
If \(|\cF|\le1\), then the assertion is trivial.
Choose two distinct lines \(L_1,L_2\in\cF\), and let \(P=L_1\cap L_2\).
If every line in \(\cF\) contains \(P\), we are done.
Otherwise choose \(L_3\in\cF\) with \(P\nleq L_3\).
Since \(L_3\) meets both \(L_1\) and \(L_2\), it lies in the \(3\)-dimensional subspace
\[
  W=L_1+L_2.
\]
Let \(L\in\cF\).
If \(P\le L\), then \(L\) must meet \(L_3\); hence \(L\le W\).
If \(P\nleq L\), then \(L\) meets \(L_1\) and \(L_2\) in two distinct points, so again \(L\le W\).
Thus all lines of \(\cF\) lie in \(W\).
\end{proof}

Shan and Zhou~\cite{SZ2024} proved a finite-vector-space \(d\)-degree Erd\H{o}s--Ko--Rado theorem, giving sharp upper bounds for the minimum degree over \(d\)-dimensional subspaces.
Their result is conceptually close to the degree viewpoint of the present paper, but the arguments below use only the Hilton--Milner theorem and the structural theorem of Ihringer and Kupavskii.

The proof of the corrected \(k+2\) theorem uses the following \(t=1,s=1\) consequence of a structural theorem of Ihringer and Kupavskii.
It is stated in the exact form needed below: the controlling family consists of points and lines, and all members not controlled by this family are placed in a small remainder.

\begin{thm}[Ihringer--Kupavskii~\cite{IK2026}]\label{thm:ik-structure}
Fix \(q\), and let \(k\ge3\), \(n\ge2k+3\).
If \(\cA\subseteq\Gr{V}{k}\) is intersecting, then there is an intersecting family \(\cC\) of points and lines such that, writing
\[
  \cA[\cC]
  =
  \{A\in\cA:C\le A\text{ for some }C\in\cC\},
\]
one has
\[
  |\cA\setminus\cA[\cC]|
  \le
  C_0(q)q^{(k-3)(n-k)+6},
\]
where
\[
  C_0(q)
  =
  \left(\frac q{q-1}\right)^4
  \left(1+\frac{q+1}{q^2-q-1}\right)^2.
\]
The constant \(C_0(q)\) depends only on \(q\), and in particular does not depend on \(k\) or \(n\).
\end{thm}

\begin{proof}[Derivation from Ihringer--Kupavskii]
In the notation of Ihringer and Kupavskii~\cite[Theorem 1.9]{IK2026}, the theorem bounds the remainder by
\[
  Cq^{(k-t)(n-k)-(s+1)(n-k-t-s-1)}.
\]
Substituting \(t=1\) and \(s=1\) gives the exponent
\[
  (k-1)(n-k)-2(n-k-3)=(k-3)(n-k)+6.
\]
Since \(n\ge2k+s+2=2k+3\), the second value of the constant in their theorem applies, namely
\[
  C=\left(\frac q{q-1}\right)^{s+3}
  \left(1+\frac{q+1}{q^2-q-1}\right)^2=C_0(q).
\]
Since \(t=1\), the nonempty members of the returned core have dimension at least \(1\), while Theorem~1.9 gives dimension at most \(s+t=2\); the core may therefore be regarded as an intersecting family of points and lines.
(The degenerate case in which the core is empty is treated as \(\cC=\emptyset\) in the proof of Theorem~\ref{thm:q-kplus2-cq3}, where it forces \(\cR=\cF\) and yields a direct bound on \(|\cF|\).)
For the parameter range used below, \(n>3k\) and \(k\ge3\), so \(n\ge2k+3\) and the displayed hypotheses are satisfied.
The next lemma records the precise inequalities involving the remainder that are used in the proof of Theorem~\ref{thm:q-kplus2-cq3}.
\end{proof}

\begin{lem}[Estimates in the range \(n>3k\)]\label{lem:cq3-estimates}
Fix \(q\).
There exists \(k_0=k_0(q)\) such that for all \(k\ge k_0\) and all \(n>3k\), with
\[
  T=\qbinom{n-2}{k-2},
  \qquad
  D=q^{k(k-2)}\qbinom{n-k-2}{k-2},
\]
we have
\[
  C_0(q)q^{(k-3)(n-k)+6}<D,
  \qquad
  C_0(q)q^{(k-3)(n-k)+6}<\frac{T}{3},
\]
and
\[
  \points{3}\qbinom{n-3}{k-3}<\frac{T}{3}.
\]
\end{lem}

\begin{proof}
We use the elementary lower bound
\[
  \qbinom{a}{b}\ge q^{b(a-b)}.
\]
Thus
\[
  T\ge q^{(k-2)(n-k)}
\]
and
\[
  D
  \ge
  q^{k(k-2)}q^{(k-2)(n-2k)}
  =
  q^{(k-2)(n-k)}.
\]
Therefore
\[
  \frac{C_0(q)q^{(k-3)(n-k)+6}}{T}
  \le
  C_0(q)q^{6-(n-k)}
  \le
  C_0(q)q^{6-2k},
\]
and the same estimate holds with \(D\) in place of \(T\).
For fixed \(q\), this is smaller than \(1/3\), and also smaller than \(1\), once \(k\) is sufficiently large.

Finally,
\[
  \frac{\qbinom{n-3}{k-3}}{\qbinom{n-2}{k-2}}
  =
  \frac{\points{k-2}}{\points{n-2}}.
\]
Hence
\[
  \points{3}\frac{\qbinom{n-3}{k-3}}{T}
  =
  \points{3}\frac{\points{k-2}}{\points{n-2}}
  \le
  \frac{\points{3}}{q-1}q^{k-n+1}
  <
  \frac{\points{3}}{q-1}q^{-2k+1},
\]
which is also smaller than \(1/3\) for all sufficiently large \(k\).
\end{proof}

\section{The \texorpdfstring{\(q\)-Frankl--Wang}{q-Frankl--Wang} bound}
\label{sec:proof-qfw}

In this section we prove Theorem~\ref{thm:q-fw}. Write
\[
  T=\qbinom{n-2}{k-2},
  \qquad
  H=\{P\le V:\dim P=1,\ d_P(\cF)>T\}.
\]
It is enough to show
\[
  |H|<\points{k}^{2},
\]
because then \(|H|\le\points{k}^{2}-1\), which is equivalent to
\[
  d_{\points{k}^{2}}(\cF)\le T.
\]

\begin{proof}[Proof of Theorem~\ref{thm:q-fw}]
We first settle \(k=2\).
Then \(T=\qbinom{n-2}{0}=1\), so a high point is a point incident with at least two lines of \(\cF\).
If all lines of \(\cF\) contain a common point \(E\), then every point \(P\ne E\) lies on at most one line of \(\cF\).
Thus the only possible high point is \(E\).
Otherwise, by Lemma~\ref{lem:line-classification}, all lines of \(\cF\) lie in a \(3\)-dimensional subspace \(W\).
Then every high point lies in \(W\), so
\[
  |H|\le \points{3}<\points{2}^{2}=\points{k}^{2}.
\]

Assume now that \(k\ge3\).
If
\[
  \bigcap_{F\in\cF}F\ne0,
\]
let \(C=\bigcap_{F\in\cF}F\).
If \(\dim C=1\), say \(C=E\), then for every point \(P\ne E\),
\[
  d_P(\cF)\le |\{F\in\Gr{V}{k}:E+P\le F\}|=\qbinom{n-2}{k-2}=T.
\]
Thus \(H\subseteq\{E\}\).
If \(\dim C\ge2\), then, writing \(c=\dim C\), for every point \(P\le C\) we have
\[
  d_P(\cF)=|\cF|\le \qbinom{n-c}{k-c}\le T.
\]
Indeed, \(\cF\subseteq\{F\in\Gr{V}{k}:C\le F\}\).
For every point \(P\nleq C\),
\[
  d_P(\cF)\le |\{F\in\Gr{V}{k}:C+P\le F\}|
  \le \qbinom{n-3}{k-3}\le T.
\]
Hence \(H=\emptyset\).

It remains to consider the case
\[
  \bigcap_{F\in\cF}F=0.
\]
By Theorem~\ref{thm:q-hm} and Lemma~\ref{lem:hm-convenient},
\[
  |\cF|\le f(n,k,q)\le \points{k}T.
\]
Now double-count incidences \((P,F)\) with \(P\in H\), \(F\in\cF\), and \(P\le F\).
Since every point in \(H\) has degree greater than \(T\),
\[
  |H|T
  <
  \sum_{P\in H}d_P(\cF).
\]
On the other hand, every \(k\)-subspace contains exactly \(\points{k}\) points, so
\[
  \sum_{P\in H}d_P(\cF)
  \le
  \points{k}|\cF|
  \le
  \points{k}^2T.
\]
Therefore
\[
  |H|<\points{k}^2.
\]
Thus \(|H|\le\points{k}^{2}-1\), since \(|H|\) is an integer, and therefore
\[
  d_{\points{k}^{2}}(\cF)\le T.
\]
Here we use the elementary equivalence \(d_r(\cF)\le T\) if and only if at most \(r-1\) points have degree greater than \(T\).
Finally,
\[
  \points{2k}
  =
  \points{k}(q^k+1),
\]
while
\[
  \points{k}=\frac{q^k-1}{q-1}<q^k+1.
\]
Thus
\[
  \points{k}^2<\points{2k},
\]
so the index \(1+\points{2k}\) is larger than \(\points{k}^{2}\).
The claimed consequence
\[
  d_{1+\points{2k}}(\cF)\le T.
\]
follows from the monotonicity of the ordered point-degrees.
\end{proof}

\section{The corrected \texorpdfstring{\(k+2\)}{k+2} bound}
\label{sec:proof-kplus2}

In this section we prove Theorem~\ref{thm:q-kplus2-cq3}.

\begin{proof}[Proof of Theorem~\ref{thm:q-kplus2-cq3}]
Choose \(k_0=k_0(q)\) as in Lemma~\ref{lem:cq3-estimates}, and increase it if necessary so that \(k_0\ge3\).
Let \(k\ge k_0\), \(n>3k\), and let \(\cF\subseteq\Gr{V}{k}\) be intersecting.
Put
\[
  T=\qbinom{n-2}{k-2},
  \qquad
  H=\{P\le V:\dim P=1,\ d_P(\cF)>T\}.
\]
It is enough to show that \(|H|\le\points{k+1}\).

By Theorem~\ref{thm:ik-structure}, there is an intersecting family \(\cC\) of points and lines such that, writing
\[
  \cR=\cF\setminus \cF[\cC],
\]
we have
\[
  |\cR|
  \le
  C_0(q)q^{(k-3)(n-k)+6}.
\]
If \(\cC=\emptyset\), then \(\cR=\cF\), and Lemma~\ref{lem:cq3-estimates} gives
\[
  d_P(\cF)\le|\cF|=|\cR|<T
\]
for every point \(P\).
Thus \(H=\emptyset\), and we are done.

Since \(\cC\) is intersecting and every member of \(\cC\) has dimension \(1\) or \(2\), the following elementary dichotomy holds.
Either all members of \(\cC\) contain a common point \(E\), or all members of \(\cC\) are lines and lie in a common \(3\)-dimensional subspace \(U\).
Indeed, if a point belongs to \(\cC\), then it is the only point of \(\cC\), and every line of \(\cC\) contains it.
If \(\cC\) consists only of pairwise intersecting lines, then either they have a common point or, by Lemma~\ref{lem:line-classification}, they lie in a \(3\)-dimensional subspace.

First suppose that all members of \(\cC\) contain a common point \(E\).
Adding \(E\) to \(\cC\), if necessary, preserves the fact that \(\cC\) is intersecting and can only decrease \(\cR\).
Thus we may assume \(E\in\cC\).
Then
\[
  \cF[\cC]\subseteq\{F\in\Gr{V}{k}:E\le F\},
\]
and no member of \(\cR\) contains \(E\).

If \(\cR=\emptyset\), then for every point \(P\ne E\),
\[
  d_P(\cF)\le |\{F\in\Gr{V}{k}:E+P\le F\}|=T,
\]
so \(H\subseteq\{E\}\).
Assume from now on that \(\cR\ne\emptyset\).

Let \(P\ne E\) be a point.
Suppose that there exists \(B\in\cR\) such that
\[
  (E+P)\cap B=0.
\]
We decompose
\[
  d_P(\cF)
  =
  d_P(\cF[\cC])+d_P(\cR).
\]
Every member of \(\cF[\cC]\) containing \(P\) contains the plane \(E+P\).
Among the \(T\) \(k\)-subspaces containing this plane, precisely
\[
  D=q^{k(k-2)}\qbinom{n-k-2}{k-2}
\]
are disjoint from \(B\).
Indeed, after quotienting by \(E+P\), this is the number of \((k-2)\)-subspaces of an \((n-2)\)-space disjoint from a fixed \(k\)-subspace; by Lemma~\ref{lem:disjoint-subspace-count}, it equals
\[
  q^{k(k-2)}\qbinom{n-k-2}{k-2}.
\]
None of these \(D\) subspaces can lie in \(\cF\), because each is disjoint from \(B\in\cF\).
Thus \(d_P(\cF[\cC])\le T-D\), while \(d_P(\cR)\le|\cR|\).
Therefore
\[
  d_P(\cF)
  \le
  T-D+|\cR|
  <
  T
\]
by Lemma~\ref{lem:cq3-estimates}.
Consequently, if \(P\in H\) and \(P\ne E\), then
\[
  (E+P)\cap B\ne0
  \qquad\text{for every }B\in\cR.
\]
Since no \(B\in\cR\) contains \(E\), this implies
\[
  P\le E+B
  \qquad\text{for every }B\in\cR.
\]
Hence
\[
  H\subseteq
  \Gr{\bigcap_{B\in\cR}(E+B)}{1}.
\]
The right-hand side is the set of points of a subspace of dimension at most \(k+1\).
Thus
\[
  |H|\le\points{k+1}.
\]

It remains to consider the second case.
Thus \(\cC\) consists of lines contained in a fixed \(3\)-dimensional subspace \(U\), and these lines have no common point.
Then every member of \(\cF[\cC]\) contains some line \(L\in\cC\), and hence meets \(U\) in dimension at least \(2\).
Let \(P\nleq U\) be a point.
The number of \(k\)-subspaces containing \(P\) and meeting \(U\) in dimension at least \(2\) is at most
\[
  \points{3}\qbinom{n-3}{k-3},
\]
because one may first choose a line \(L\le U\) and then count the \(k\)-subspaces containing \(P+L\).
This is only an upper bound, since the same \(k\)-subspace may contain more than one line of \(U\).
Adding the remainder gives
\[
  d_P(\cF)
  \le
  \points{3}\qbinom{n-3}{k-3}+|\cR|
  <
  T
\]
by Lemma~\ref{lem:cq3-estimates}.
So every high point lies in \(U\), and hence
\[
  |H|\le\points{3}\le\points{k+1}.
\]

In all cases \(|H|\le\points{k+1}\).
Equivalently,
\[
  d_{1+\points{k+1}}(\cF)\le T=\qbinom{n-2}{k-2}.
\]
\end{proof}

\section{Extremal examples and \texorpdfstring{large-\(i\)}{large-i} reductions}
\label{sec:constructions}

In this section we give the sharpness constructions that calibrate the degree indices used in the introduction, and we prove two necessary conditions that any counterexample to Conjecture~\ref{conj:q-large-i-corrected} must satisfy.
First, a vector-space Hilton--Milner family shows that the naive \(q\)-analogue of the set-theoretic \(k+2\) theorem has the wrong index and that the corrected index \(1+\points{k+1}\) is sharp.
Second, the saturated Frankl--Hilton--Milner family \(\cG_i\) shows that the second term in Conjecture~\ref{conj:q-large-i-corrected} must be \(q^{k-1}\qbinom{n-i-1}{k-i}\).
Finally, two necessary conditions on a strict counterexample to the large-\(i\) conjecture are established: such a family must either have very large diversity~\cite{Kupavskii2018diversity} outside a maximum point-star or exhibit strong localization of its noncentral high points.

\subsection{The naive \texorpdfstring{\(k+2\)}{k+2} analogue and sharpness}

The same phenomenon that invalidates the literal \(2k+1\) index also affects the set-theoretic \(k+2\) theorem of Huang and Rao.
The most direct \(q\)-analogue would replace \(k+2\) by \(\points{k}+2\), but this is still too small.

\begin{prop}[The naive \(q\)-analog of the \(k+2\) theorem is false]\label{prop:q-kplus2-false}
For every \(k\ge2\) and every \(n\ge k+1\), there is an intersecting family \(\cF\subseteq\Gr{V}{k}\) such that
\[
  d_{\points{k}+2}(\cF)
  >
  \qbinom{n-2}{k-2}.
\]
Consequently no constant \(C_q\) can make the bound
\[
  d_{\points{k}+2}(\cF)\le \qbinom{n-2}{k-2}
\]
valid for all \(n>C_q\,k\).
\end{prop}

\begin{proof}[Proof of Proposition~\ref{prop:q-kplus2-false}]
Choose a point \(E\le V\) and a \(k\)-subspace \(L\le V\) with \(E\nleq L\), and put \(W=E+L\).
Thus \(\dim W=k+1\).
Consider the standard vector-space Hilton--Milner family
\[
  \cM(E,L)
  =
  \{F\in\Gr{V}{k}:E\le F,\ \dim(F\cap L)\ge1\}
  \cup
  \{G\in\Gr{W}{k}:E\nleq G\}.
\]
It is intersecting.
Indeed, two members of the first part both contain \(E\), two members of the second part are hyperplanes of the \((k+1)\)-space \(W\) and hence meet, and if \(F\) is in the first part while \(G\) is in the second part, then \(F\) contains \(E+Q\) for some point \(Q\le L\), so
\[
  \dim(G\cap(E+Q))\ge \dim G+\dim(E+Q)-\dim W=k+2-(k+1)=1.
\]

Let \(T=\qbinom{n-2}{k-2}\).
We claim that every point \(P\le W\) has degree larger than \(T\) in \(\cM(E,L)\).
For \(P=E\), fix a point \(Q\le L\).
All \(k\)-subspaces containing the plane \(E+Q\) belong to the first part of \(\cM(E,L)\), giving \(T\) members.
Choose another point \(Q'\le L\) with \(Q'\ne Q\), which exists because \(k\ge2\).
There is a \(k\)-subspace \(F\le V\) such that \(E+Q'\le F\) and \(Q\nleq F\); for instance, take a hyperplane of the \((k+1)\)-space \(W=E+L\) containing \(E+Q'\) but not \(Q\).
Then \(F\) belongs to the first part of \(\cM(E,L)\), but it is not among the \(T\) subspaces containing \(E+Q\).
Thus \(d_E(\cM(E,L))>T\).
Now let \(P\le W\) with \(P\ne E\).
The plane \(E+P\) meets \(L\) in a point, so every \(k\)-subspace of \(V\) containing \(E+P\) belongs to the first part of \(\cM(E,L)\).
There are exactly \(T\) such \(k\)-subspaces.
In addition, the hyperplanes of \(W\) containing \(P\) but not \(E\) contribute new members from the second part.
Their number is
\[
  \points{k}-\points{k-1}
  =
  q^{k-1}.
\]
Hence
\[
  d_P(\cM(E,L))
  =
  T+q^{k-1}
  >
  T
  \qquad\text{for every }P\le W,\ P\ne E.
\]
Thus all \(\points{k+1}\) points of \(W\) have degree greater than \(T\).
Since
\[
  \points{k+1}
  =
  \points{k}+q^k
  \ge
  \points{k}+2
  \qquad (k\ge2),
\]
the \((\points{k}+2)\)-nd largest point-degree is still greater than \(T\).
\end{proof}

\begin{cor}[Sharpness of the corrected \(k+2\) index]\label{cor:kplus2-index-sharp}
For every \(k\ge2\) and every \(n\ge k+1\), there exists an intersecting family \(\cF\subseteq\Gr{V}{k}\) such that
\[
  d_{\points{k+1}}(\cF)
  >
  \qbinom{n-2}{k-2}.
\]
Hence the index \(1+\points{k+1}\) in Theorem~\ref{thm:q-kplus2-cq3} cannot be replaced by \(\points{k+1}\).
\end{cor}

\begin{proof}[Proof of Corollary~\ref{cor:kplus2-index-sharp}]
The family \(\cM(E,L)\) constructed in the proof of Proposition~\ref{prop:q-kplus2-false} has all \(\points{k+1}\) points of the \((k+1)\)-space \(W=E+L\) of degree greater than \(T=\qbinom{n-2}{k-2}\).
Therefore
\[
  d_{\points{k+1}}(\cM(E,L))>T,
\]
which proves the assertion.
\end{proof}

\subsection{The saturated Frankl--Hilton--Milner family}

For larger degree order statistics, the unsaturated Hilton--Milner extrapolation does not give the sharp bound.
The right construction is the following Frankl degree-type family.

\begin{defn}[The saturated Frankl--Hilton--Milner family]\label{def:Gi}
Let \(2\le i\le k\).
Fix a point \(X\le V\) and an \(i\)-dimensional subspace \(Y_i\le V\) with \(X\cap Y_i=0\).
Put
\[
  \cG_i=\cG_i^\Delta\cup\cG_i^\gamma,
\]
where
\[
  \cG_i^\Delta
  =
  \{F\in\Gr{V}{k}:X\le F,\ \dim(F\cap Y_i)\ge1\},
\]
and
\[
  \cG_i^\gamma
  =
  \{G\in\Gr{V}{k}:Y_i\le G+X,\ G\cap X=0\}.
\]
\end{defn}

\begin{prop}[Point degrees in \(\cG_i\)]\label{prop:Gi-point-degrees}
Let \(2\le i\le k\), and let \(\cG_i\) be as in Definition~\ref{def:Gi}.
Then \(\cG_i\) is intersecting.
Moreover, for every point \(P\le X+Y_i\) with \(P\ne X\),
\[
  d_P(\cG_i)
  =
  \qbinom{n-2}{k-2}
  +
  q^{k-1}\qbinom{n-i-1}{k-i}.
\]
\end{prop}

\begin{proof}[Proof of Proposition~\ref{prop:Gi-point-degrees}]
The family \(\cG_i^\Delta\) is a subfamily of the \(X\)-star, so it is intersecting.
If \(G,G'\in\cG_i^\gamma\), then \(G+X\) and \(G'+X\) both contain \(Y_i\).
Since \(G\cap X=G'\cap X=0\), each of \(G\) and \(G'\) maps onto \(Y_i\) under the quotient map \(V\to V/X\).
In particular, \(G\) and \(G'\) each meet \(X+Y_i\) in an \(i\)-subspace complementary to \(X\); two such \(i\)-subspaces inside the \((i+1)\)-space \(X+Y_i\) meet in dimension at least \(i-1\), which is nonzero for \(i\ge2\).
Thus \(G\cap G'\ne0\).
Finally, if \(F\in\cG_i^\Delta\) and \(G\in\cG_i^\gamma\), then \(F\) contains \(X\) and also a point of \(Y_i\).
Hence \(F\cap(X+G)\) has dimension at least \(2\).
Since \((X+G)/G\) is one-dimensional, this forces \(F\cap G\ne0\).

Now fix a point \(P\le X+Y_i\), \(P\ne X\).
The contribution from \(\cG_i^\Delta\) is exactly the number of \(k\)-subspaces containing the plane \(X+P\), namely
\[
  \qbinom{n-2}{k-2}.
\]
Indeed, any such \(k\)-subspace belongs to \(\cG_i^\Delta\), and every member of \(\cG_i^\Delta\) containing \(P\) contains \(X+P\).

It remains to count the contribution from \(\cG_i^\gamma\).
Let \(\pi:V\to V/X\) be the quotient map.
A member \(G\in\cG_i^\gamma\) corresponds to a \(k\)-subspace \(\overline G=\pi(G)\) of \(V/X\) containing \(\pi(Y_i)\), together with a choice of a complement \(G\) to \(X\) inside \(\pi^{-1}(\overline G)\).
There are
\[
  \qbinom{n-i-1}{k-i}
\]
choices for \(\overline G\).
For a fixed \(\overline G\), the complements to \(X\) in \(\pi^{-1}(\overline G)\) are graphs of linear maps \(\overline G\to X\), hence there are \(q^k\) of them.
The additional requirement \(P\le G\) fixes the value of this linear map on the one-dimensional subspace \(\pi(P)\), and leaves \(q^{k-1}\) choices.
Thus
\[
  |\{G\in\cG_i^\gamma:P\le G\}|
  =
  q^{k-1}\qbinom{n-i-1}{k-i},
\]
which gives the stated degree formula.
\end{proof}

\begin{cor}[Sharpness of the large-\(i\) bound]\label{cor:Gi-sharpness}
Let \(2\le i\le k\), let \(n\ge2k+1\), and let \(\cG_i\) be as in Definition~\ref{def:Gi}.
Put
\[
  T=\qbinom{n-2}{k-2},
  \qquad
  Q_i=q^{k-1}\qbinom{n-i-1}{k-i}.
\]
Then
\[
  d_{\points{i+1}}(\cG_i)=T+Q_i.
\]
\end{cor}

\begin{proof}[Proof of Corollary~\ref{cor:Gi-sharpness}]
Let \(Z=X+Y_i\).
By Proposition~\ref{prop:Gi-point-degrees}, every point \(P\le Z\) with \(P\ne X\) has degree \(T+Q_i\) in \(\cG_i\).
There are
\[
  \points{i+1}-1=q\points{i}
\]
such points.

We next check the degree of \(X\).
Choose a \(2\)-dimensional subspace \(Y_2\le Y_i\).
The members of \(\cG_i\) containing \(X\) are precisely the members of \(\cG_i^\Delta\), and they include all \(k\)-subspaces containing \(X\) and meeting \(Y_2\).
In the quotient \(V/X\), this latter number is the number of \((k-1)\)-subspaces meeting a fixed \(2\)-subspace.
Counting according to whether the intersection has dimension \(1\) or \(2\), and using the convention \(\qbinom{a}{b}=0\) for \(b<0\), this number is
\[
  \points{2}q^{k-2}\qbinom{n-3}{k-2}
  +
  \qbinom{n-3}{k-3}.
\]
Using \(\points{2}=q+1\) and the Gaussian recurrence
\[
  \qbinom{n-2}{k-2}
  =
  \qbinom{n-3}{k-3}
  +
  q^{k-2}\qbinom{n-3}{k-2},
\]
we get
\[
  d_X(\cG_i)
  \ge
  T+q^{k-1}\qbinom{n-3}{k-2}.
\]
On the other hand, in \(V/X\), the number of \(k\)-subspaces containing the fixed \(i\)-subspace \(\pi(Y_i)\) is
\[
  \qbinom{n-i-1}{k-i},
\]
while the number of \(k\)-subspaces containing the fixed \(2\)-subspace \(\pi(Y_2)\) is
\[
  \qbinom{n-3}{k-2}.
\]
Since \(\pi(Y_2)\le\pi(Y_i)\), the former number is at most the latter:
\[
  \qbinom{n-i-1}{k-i}\le \qbinom{n-3}{k-2}.
\]
Therefore \(d_X(\cG_i)\ge T+Q_i\).

It remains to show that no point outside \(Z\) has degree as large as \(T+Q_i\).
Let \(P\nleq Z\).
Then the plane \(X+P\) is disjoint from \(Y_i\).
The contribution from \(\cG_i^\Delta\) is the number of \(k\)-subspaces containing \(X+P\) and meeting \(Y_i\), namely
\[
  \qbinom{n-2}{k-2}
  -
  q^{i(k-2)}\qbinom{n-i-2}{k-2}
  <
  T.
\]
Here the subtracted term is counted by applying Lemma~\ref{lem:disjoint-subspace-count} in the quotient by \(X+P\).
For the \(\gamma\)-part, if \(i<k\), then a member \(G\in\cG_i^\gamma\) containing \(P\) projects to a \(k\)-subspace of \(V/X\) containing both \(\pi(Y_i)\) and \(\pi(P)\), and the complement map is fixed on \(\pi(P)\).
Thus this contribution is at most
\[
  q^{k-1}\qbinom{n-i-2}{k-i-1}
  <
  q^{k-1}\qbinom{n-i-1}{k-i}
  =
  Q_i.
\]
If \(i=k\), the \(\gamma\)-part contributes nothing for \(P\nleq Z\).
Hence \(d_P(\cG_i)<T+Q_i\) for every \(P\nleq Z\).

There are exactly \(\points{i+1}\) points in \(Z\): the point \(X\), whose degree is at least \(T+Q_i\), and the other \(\points{i+1}-1\) points, whose degrees are exactly \(T+Q_i\).
All points outside \(Z\) have smaller degree.
The point \(X\) may have degree larger than \(T+Q_i\), but this does not affect the \(\points{i+1}\)-th largest point-degree.
Thus
\[
  d_{\points{i+1}}(\cG_i)=T+Q_i.
\]
\end{proof}

The construction is due to Ihringer and Kupavskii~\cite{IK2026}; we have included the intersecting verification, point-degree calculation, and sharpness check above for completeness.
Proposition~\ref{prop:Gi-point-degrees} shows that the previously natural guess
\[
  \qbinom{n-2}{k-2}
  +
  q^{k-i}\qbinom{n-i-1}{k-i}
\]
is too small by a factor \(q^{i-1}\) in the second term.
The source of this factor is that, in the \(\gamma\)-part of \(\cG_i\), fixing the quotient \(k\)-subspace still leaves \(q^k\) complement maps, and requiring a member to pass through a specified point fixes only one one-dimensional value, leaving \(q^{k-1}\) choices.

\subsection{Necessary conditions on \texorpdfstring{large-\(i\)}{large-i} counterexamples}

We do not prove Conjecture~\ref{conj:q-large-i-corrected} in full.
The case \(i=2\) is tied to the corrected \(k+2\) phenomenon above; we state the conditions below for \(i\ge3\).
The following proposition records a first incidence reduction.
It says that if the conjectured degree bound fails, then either the non-star part already exceeds the lower diversity threshold in the cross-intersecting theorem of Ihringer and Kupavskii~\cite[Theorem 4.7]{IK2026}, or the high points are forced to cluster inside a single non-star member.
For a subspace \(B\le V\) and a set \(H\) of points, we write \(|B\cap H|\) for the number of points \(P\in H\) with \(P\le B\).

\begin{prop}[Diversity-or-clustering obstruction]\label{prop:anti-clustering-reduction}
Let \(3\le i\le k\), let \(V\) be \(n\)-dimensional with \(n\ge k+1\), let \(\cF\subseteq\Gr{V}{k}\) be intersecting, and put
\[
  T=\qbinom{n-2}{k-2},
  \qquad
  Q_i=q^{k-1}\qbinom{n-i-1}{k-i},
  \qquad
  \Gamma_i=qQ_i=q^k\qbinom{n-i-1}{k-i}.
\]
Let \(X\le V\) be a point of maximum degree in \(\cF\), and write
\[
  \cB=\cF\setminus\cF[X].
\]
Assume that
\[
  d_{\points{i+1}}(\cF)>T+Q_i.
\]
Let
\[
  S_X=\{P\le V:\dim P=1,\ P\ne X,\ d_P(\cF)>T+Q_i\}.
\]
Then \(|S_X|\ge\points{i+1}-1=q\points{i}\).
Moreover, for every subset \(H\subseteq S_X\) with
\[
  |H|=\points{i+1}-1=q\points{i},
\]
at least one of the following alternatives holds:
\[
  |\cB|>\Gamma_i,
\]
or there is a member \(B_0\in\cB\) such that
\[
  |B_0\cap H|>\points{i}.
\]
\end{prop}

\begin{proof}[Proof of Proposition~\ref{prop:anti-clustering-reduction}]
Since \(d_{\points{i+1}}(\cF)>T+Q_i\), there are at least \(\points{i+1}\) points of \(V\) whose degrees in \(\cF\) are larger than \(T+Q_i\).
The point \(X\) has maximum degree, so \(X\) is one of these high points.
Consequently, after removing \(X\), there remain at least
\[
  \points{i+1}-1=q\points{i}
\]
points \(P\ne X\) with \(d_P(\cF)>T+Q_i\).
Thus \(|S_X|\ge q\points{i}\), as claimed.

Let \(H\subseteq S_X\) with \(|H|=q\points{i}\).
For each \(P\in H\), the members of the star part \(\cF[X]\) that also contain \(P\) must contain the plane \(X+P\).
Hence
\[
  d_P(\cF[X])
  \le
  \qbinom{n-2}{k-2}
  =
  T.
\]
Since \(P\in S_X\), it follows that
\[
  d_P(\cB)
  =
  d_P(\cF)-d_P(\cF[X])
  >
  (T+Q_i)-T
  =
  Q_i.
\]

Now double-count incidences \((P,B)\) with \(P\in H\), \(B\in\cB\), and \(P\le B\).
The strict inequality in the next display comes from the strict counterexample assumption \(d_{\points{i+1}}(\cF)>T+Q_i\), which gave \(d_P(\cB)>Q_i\) for every \(P\in H\).
The preceding inequality gives
\[
  \sum_{P\in H}d_P(\cB)
  >
  |H|Q_i
  =
  q\points{i}Q_i
  =
  \points{i}\Gamma_i.
\]
On the other hand,
\[
  \sum_{P\in H}d_P(\cB)
  =
  \sum_{B\in\cB}|B\cap H|.
\]
Suppose that the first alternative fails, so that \(|\cB|\le\Gamma_i\), and that the second alternative also fails, so that \(|B\cap H|\le\points{i}\) for every \(B\in\cB\).
Then
\[
  \sum_{B\in\cB}|B\cap H|
  \le
  \points{i}|\cB|
  \le
  \points{i}\Gamma_i,
\]
contradicting the strict lower bound above.
Therefore at least one of the two stated alternatives must hold.
\end{proof}

We next strengthen the reduction in the bounded-diversity range.

\begin{prop}[One-member loss reduction]\label{prop:one-member-loss-reduction}
Let \(3\le i\le k\), let \(n\ge2k\), and let \(\cF\subseteq\Gr{V}{k}\) be intersecting.
Put
\[
  T=\qbinom{n-2}{k-2},
  \qquad
  Q_i=q^{k-1}\qbinom{n-i-1}{k-i},
  \qquad
  R_{n,k}=q^{k(k-2)}\qbinom{n-k-2}{k-2}.
\]
Let \(X\le V\) be a point of maximum degree in \(\cF\), and write
\[
  \cB=\cF\setminus\cF[X].
\]
If
\[
  |\cB|\le R_{n,k}+Q_i,
\]
then
\[
  d_{\points{i+1}}(\cF)\le T+Q_i.
\]
Equivalently, every strict counterexample to Conjecture~\ref{conj:q-large-i-corrected} satisfies
\[
  |\cF\setminus\cF[X]|>R_{n,k}+Q_i.
\]
\end{prop}

\begin{proof}[Proof of Proposition~\ref{prop:one-member-loss-reduction}]
Assume for contradiction that
\[
  d_{\points{i+1}}(\cF)>T+Q_i.
\]
Let
\[
  S_X=\{P\le V:\dim P=1,\ P\ne X,\ d_P(\cF)>T+Q_i\}.
\]
As in the proof of Proposition~\ref{prop:anti-clustering-reduction}, we have
\[
  |S_X|\ge q\points{i}
\]
and, for every \(P\in S_X\),
\[
  d_P(\cB)>Q_i.
\]

We first prove a one-member loss estimate.
Let \(P\in S_X\), and suppose that there is \(B_0\in\cB\) such that
\[
  P\nleq X+B_0.
\]
Then the plane \(L=X+P\) is disjoint from \(B_0\).
Indeed, if \(0\ne v\in L\cap B_0\), then \(v\notin X\) because \(B_0\) avoids \(X\), and hence \(P\le X+v\le X+B_0\), a contradiction.

Among all \(k\)-subspaces containing \(L\), exactly
\[
  R_{n,k}
  =
  q^{k(k-2)}\qbinom{n-k-2}{k-2}
\]
are disjoint from \(B_0\).
To see this, quotient by \(L\).
The image of \(B_0\) is a \(k\)-subspace of \(V/L\), and we need to choose a \((k-2)\)-subspace of \(V/L\) disjoint from it.
Lemma~\ref{lem:disjoint-subspace-count}, with \(N=n-2\), \(m=k\), and \(r=k-2\), gives the displayed number.
None of these \(R_{n,k}\) subspaces can belong to \(\cF[X]\), since each is disjoint from \(B_0\in\cF\).
Therefore
\[
  d_P(\cF[X])\le T-R_{n,k}.
\]
Using the assumed bound \(|\cB|\le R_{n,k}+Q_i\), we obtain
\[
  d_P(\cF)
  =
  d_P(\cF[X])+d_P(\cB)
  \le
  T-R_{n,k}+|\cB|
  \le
  T+Q_i,
\]
contradicting \(P\in S_X\).
Consequently,
\[
  P\le X+B
  \qquad\text{for every }P\in S_X\text{ and every }B\in\cB.
\]

Let \(\pi:V\to V/X\) be the quotient map, and put
\[
  U=\langle \pi(P):P\in S_X\rangle\le V/X.
\]
The preceding conclusion implies
\[
  U\le \pi(B)
  \qquad\text{for every }B\in\cB.
\]
Write \(s=\dim U\).
The preimage \(\pi^{-1}(U)\) has dimension \(s+1\) and contains \(X\), so it contains exactly \(\points{s+1}-1=q\points{s}\) points other than \(X\).
Equivalently, each point of \(U\) has \(q\) point-lifts outside \(X\).
Since \(|S_X|\ge q\points{i}\), we must have \(s\ge i\): if \(s\le i-1\), then the number of points \(P\ne X\) with \(\pi(P)\le U\) is at most
\[
  q\points{s}\le q\points{i-1}<q\points{i},
\]
contradicting the definition of \(U\) and the lower bound on \(|S_X|\).

Fix \(P\in S_X\).
Every member \(B\in\cB\) counted by \(d_P(\cB)\) has \(\pi(B)\) a \(k\)-subspace of \(V/X\) containing \(U\).
There are \(\qbinom{n-s-1}{k-s}\) choices for this quotient \(k\)-subspace.
For each such quotient, the members \(B\) avoiding \(X\) are complements to \(X\), and the additional condition \(P\le B\) fixes the complement map on the point \(\pi(P)\); hence there are at most \(q^{k-1}\) choices.
Thus
\[
  d_P(\cB)
  \le
  q^{k-1}\qbinom{n-s-1}{k-s}
  \le
  q^{k-1}\qbinom{n-i-1}{k-i}
  =
  Q_i,
\]
where the second inequality follows because \(s\ge i\).
This contradicts \(d_P(\cB)>Q_i\).
The contradiction proves the proposition.
\end{proof}

\begin{rmk}
Propositions~\ref{prop:anti-clustering-reduction} and~\ref{prop:one-member-loss-reduction} should not be read as a proof of Conjecture~\ref{conj:q-large-i-corrected}.
They isolate the two obstructions that any strict counterexample must overcome: sufficiently large diversity outside a maximum point-star, and a strong localization of many noncentral high points.
\end{rmk}

The same incidence count gives a weaker but unconditional estimate in terms of the diversity.

\begin{cor}[Diversity bound]\label{cor:crude-large-i-bound}
Let \(3\le i\le k\), let \(V\) be \(n\)-dimensional with \(n\ge k+1\), let \(\cF\subseteq\Gr{V}{k}\) be intersecting, and put
\[
  T=\qbinom{n-2}{k-2}.
\]
Let \(X\le V\) be a point of maximum degree in \(\cF\), and write \(\cB=\cF\setminus\cF[X]\).
Then
\[
  d_{\points{i+1}}(\cF)
  \le
  T+
  \frac{\points{k}}{q\points{i}}|\cB|.
\]
Consequently, if
\[
  |\cB|\le q^k\qbinom{n-i-1}{k-i},
\]
then
\[
  d_{\points{i+1}}(\cF)
  \le
  T+
  \frac{\points{k}}{\points{i}}\,
  q^{k-1}\qbinom{n-i-1}{k-i}.
\]
\end{cor}

\begin{proof}
Put \(D=d_{\points{i+1}}(\cF)\).
If \(D\le T\), there is nothing to prove.
Assume \(D>T\).
There are at least \(\points{i+1}\) points of \(V\) with degree at least \(D\).
Since \(X\) is a maximum-degree point, after excluding \(X\) we can choose a set \(H\) of noncentral points such that
\[
  |H|=\points{i+1}-1=q\points{i}
  \qquad\text{and}\qquad
  d_P(\cF)\ge D\quad(P\in H).
\]
As in the proof of Proposition~\ref{prop:anti-clustering-reduction}, for every \(P\in H\) we have \(d_P(\cF[X])\le T\), and hence
\[
  d_P(\cB)\ge D-T.
\]
Double-counting incidences \((P,B)\) with \(P\in H\), \(B\in\cB\), and \(P\le B\), we get
\[
  q\points{i}(D-T)
  \le
  \sum_{P\in H}d_P(\cB)
  =
  \sum_{B\in\cB}|B\cap H|
  \le
  \points{k}|\cB|,
\]
because each \(k\)-subspace contains exactly \(\points{k}\) points.
This proves the first inequality.
The second follows by substituting the stated upper bound on \(|\cB|\).
\end{proof}

The clustering alternative is impossible in one useful quotient-span case.

\begin{cor}[Low quotient-span case]\label{cor:quotient-span-case}
Let \(3\le i\le k\), let \(V\) be \(n\)-dimensional with \(n\ge k+1\), let \(\cF\subseteq\Gr{V}{k}\) be intersecting, and put
\[
  T=\qbinom{n-2}{k-2},
  \qquad
  Q_i=q^{k-1}\qbinom{n-i-1}{k-i},
  \qquad
  \Gamma_i=qQ_i.
\]
Let \(X\le V\) be a point of maximum degree in \(\cF\), write \(\cB=\cF\setminus\cF[X]\), and let \(\pi:V\to V/X\) be the quotient map.
Put
\[
  S_X=\{P\le V:\dim P=1,\ P\ne X,\ d_P(\cF)>T+Q_i\}.
\]
If
\[
  |\cB|\le \Gamma_i
\]
and
\[
  \dim\langle \pi(P):P\in S_X\rangle\le i,
\]
then
\[
  d_{\points{i+1}}(\cF)\le T+Q_i.
\]
\end{cor}

\begin{proof}
Assume for contradiction that \(d_{\points{i+1}}(\cF)>T+Q_i\).
By Proposition~\ref{prop:anti-clustering-reduction}, the set \(S_X\) has size at least \(q\points{i}\).
Choose \(H\subseteq S_X\) with \(|H|=q\points{i}\).
Since \(\langle\pi(S_X)\rangle\) has dimension at most \(i\), the set \(\pi(H)\) is contained in an \(i\)-dimensional subspace of \(V/X\), and therefore
\[
  |\pi(H)|\le\points{i}.
\]
For every \(B\in\cB\), we have \(X\nleq B\), so the quotient map is injective on the set of points of \(B\).
Consequently
\[
  |B\cap H|\le |\pi(H)|\le\points{i}
  \qquad\text{for every }B\in\cB.
\]
Thus the clustering alternative in Proposition~\ref{prop:anti-clustering-reduction} cannot occur.
The diversity alternative cannot occur either, because \(|\cB|\le\Gamma_i\) by assumption.
This contradiction proves the desired degree bound.
\end{proof}

\section{Concluding remarks}
\label{sec:conclusion}

The proof of Theorem~\ref{thm:q-fw} shows that the \(q\)-Frankl--Wang degree theorem follows formally from a sufficiently strong nontrivial-intersecting bound.
More precisely, for \(k\ge3\), it is enough to know that every intersecting \(\cF\subseteq\Gr{V}{k}\) with zero total intersection satisfies
\[
  |\cF|\le \points{k}\qbinom{n-2}{k-2}.
\]
The vector-space Hilton--Milner theorem supplies this inequality for all \(n\ge2k+1\): the original theorem of Blokhuis et al.~\cite{Blokhuis2010} gives all cases except \(q=2,n=2k+1\), and Wang--Xu--Zhang~\cite{WXZ2023} supplies this boundary case.
Thus no shifting~\cite{Frankl1987shifting} or spectral~\cite{HZ2017} replacement is needed for Theorem~\ref{thm:q-fw}; the order-statistic statement is a direct consequence of the nontrivial-intersecting theorem and the elementary high-point incidence count.

The situation for Conjecture~\ref{conj:q-large-i-corrected} is different.
Suppose \(d_{\points{i+1}}(\cF)>T+Q_i\), where \(T=\qbinom{n-2}{k-2}\) and \(Q_i=q^{k-1}\qbinom{n-i-1}{k-i}\).
After choosing a maximum-degree point \(X\), there are at least \(\points{i+1}-1\) points \(P\ne X\) with \(d_P(\cF\setminus\cF[X])>Q_i\), because \(d_P(\cF[X])\le T\).
Proposition~\ref{prop:anti-clustering-reduction} shows that the failure of the conjectured bound has only two possible causes.
Either the diversity \(|\cF\setminus\cF[X]|\) already exceeds
\[
  q^k\qbinom{n-i-1}{k-i},
\]
which is exactly the lower threshold in the cross-intersecting theorem of Ihringer and Kupavskii~\cite[Theorem 4.7]{IK2026}, or the noncentral high points are genuinely clustered: some member of \(\cF\setminus\cF[X]\) contains more than \(\points{i}\) of them.
Thus, after Corollary~\ref{cor:quotient-span-case}, the difficult case is the following quotient-spanning one: if \(H\) is a set of \(\points{i+1}-1\) noncentral high points and \(\pi:V\to V/X\) is the quotient map, then
\[
  \dim\langle \pi(P):P\in H\rangle>i.
\]
Proposition~\ref{prop:one-member-loss-reduction} sharpens this picture.
It shows that a strict counterexample must have
\[
  |\cF\setminus\cF[X]|
  >
  R_{n,k}+Q_i,
  \qquad
  R_{n,k}=q^{k(k-2)}\qbinom{n-k-2}{k-2}.
\]
Thus the bounded-diversity clustered case is eliminated by the one-member loss estimate.
Moreover, standard Gaussian estimates show that for fixed \(q\), fixed \(\varepsilon>0\), \(i\ge\varepsilon k\), and \(n>C\,k\) with \(C>2\), one has
\[
  q^k\qbinom{n-i-1}{k-i}
  \ll
  R_{n,k}
\]
as \(k\to\infty\).
More explicitly, standard bounds for Gaussian coefficients give
\[
  \log_q
  \frac{R_{n,k}}
       {q^k\qbinom{n-i-1}{k-i}}
  \ge
  (i-2)(n-k)-i-O_q(1).
\]
The remaining task toward Conjecture~\ref{conj:q-large-i-corrected} is therefore a very-large-diversity localization theorem: when \(|\cF\setminus\cF[X]|>R_{n,k}+Q_i\), one should prove that the noncentral high points are forced into a bounded low-dimensional quotient core, specifically a \(\cG_2\)-type core, and hence are too few to violate the \(\points{i+1}\)-th degree bound for \(i\) linear in \(k\).

\subsection*{Open problems}

We close with several concrete problems suggested by the results above.

\begin{ques}[The exact number of high points]\label{ques:high-point-number}
For \(n\ge2k+1\), determine
\[
  h_q(n,k)
  =
  \max_{\cF}
  \left|
  \left\{
    P\le V:\dim P=1,\ 
    d_P(\cF)>\qbinom{n-2}{k-2}
  \right\}
  \right|,
\]
where the maximum is over all intersecting families \(\cF\subseteq\Gr{V}{k}\).
The Hilton--Milner construction gives \(h_q(n,k)\ge\points{k+1}\), while Theorem~\ref{thm:q-fw} gives
\[
  h_q(n,k)\le \points{k}^{2}-1.
\]
Moreover, Theorem~\ref{thm:q-kplus2-cq3} and Corollary~\ref{cor:kplus2-index-sharp} imply that, for fixed \(q\), sufficiently large \(k\), and \(n>3k\), one has
\[
  h_q(n,k)=\points{k+1}.
\]
Is this equality true in the full range \(n\ge2k+1\), or at least in a range \(n\ge2k+M_q\)?
\end{ques}

\begin{ques}[The optimal range for the corrected \(k+2\) theorem]\label{ques:kplus2-range}
For fixed \(q\), determine the smallest additive threshold \(M_q\), if it exists, such that for all sufficiently large \(k\), the condition \(n\ge2k+M_q\) implies
\[
  d_{1+\points{k+1}}(\cF)
  \le
  \qbinom{n-2}{k-2}
\]
for every intersecting family \(\cF\subseteq\Gr{V}{k}\).
Theorem~\ref{thm:q-kplus2-cq3} proves the weaker linear range \(n>3k\).
\end{ques}

\begin{ques}[Large-\(i\) localization]\label{ques:large-i-localization}
For fixed \(q\) and \(\varepsilon>0\), prove or disprove the following localization principle.
There is a constant \(C_{\varepsilon,q}\) such that, whenever \(i\ge\varepsilon k\), \(n>C_{\varepsilon,q}\,k\), and \(\cF\subseteq\Gr{V}{k}\) is intersecting, every strict counterexample to Conjecture~\ref{conj:q-large-i-corrected} has a maximum-degree point \(X\) for which the noncentral high points are forced into a bounded low-dimensional quotient core in \(V/X\).
An effective form of such a theorem, combined with Proposition~\ref{prop:one-member-loss-reduction}, would prove the large-\(i\) conjecture.
\end{ques}

\begin{ques}[\(r\)-subspace degree order statistics]\label{ques:r-subspace-degrees}
Let \(1\le r<k\), and for an \(r\)-subspace \(S\le V\) define
\[
  d_S(\cF)=|\{F\in\cF:S\le F\}|.
\]
In a point-star, every \(r\)-subspace \(S\) not containing the center has degree
\[
  \qbinom{n-r-1}{k-r-1}.
\]
How many \(r\)-subspaces can have degree larger than this threshold in an intersecting family?
This would be an order-statistic counterpart to the finite-vector-space \(d\)-degree Erd\H{o}s--Ko--Rado theorem of Shan and Zhou~\cite{SZ2024}.
\end{ques}

\bibliographystyle{plain}
\bibliography{main}

@article{Blokhuis2010,
  author = {Blokhuis, Aart and Brouwer, Andries E. and Chowdhury, Ameera and Frankl, P{\'e}ter and Mussche, Tibor and Patk{\'o}s, Bal{\'a}zs and Sz{\H o}nyi, Tam{\'a}s},
  title = {A {H}ilton-{M}ilner theorem for vector spaces},
  journal = {Electronic Journal of Combinatorics},
  volume = {17},
  number = {1},
  year = {2010},
  pages = {Research Paper R71},
  doi = {10.37236/343},
  url = {https://www.emis.de/journals/EJC/Volume_17/PDF/v17i1r71.pdf}
}

@article{EKR1961,
    author = {Erd\H{o}s, P. and Ko, C. and Rado, R.},
    title = {{Intersection Theorems for Systems of Finite Sets}},
    journal = {Quart. J. Math.},
    volume = {12},
    number = {1},
    pages = {313--320},
    year = {1961},
    doi = {10.1093/qmath/12.1.313},
}

@article{FW2025,
  author = {Frankl, P{\'e}ter and Wang, Jian},
  title = {On the largest degrees in intersecting hypergraphs},
  journal = {arXiv preprint arXiv:2511.15508},
  year = {2025},
  eprint = {2511.15508},
  archiveprefix = {arXiv},
  primaryclass = {math.CO},
  doi = {10.48550/arXiv.2511.15508},
  url = {https://arxiv.org/abs/2511.15508}
}

@article{Hsieh1975,
  author = {Hsieh, Wen Ning},
  title = {Intersection theorems for systems of finite vector spaces},
  journal = {Discrete Math.},
  volume = {12},
  year = {1975},
  pages = {1--16}
}

@article{HR2026,
  author = {Huang, Hao and Rao, Rui},
  title = {On the $\ell$-th largest degree of an intersecting family},
  journal = {arXiv preprint arXiv:2602.01692v3},
  year = {2026},
  eprint = {2602.01692},
  archiveprefix = {arXiv},
  primaryclass = {math.CO},
  doi = {10.48550/arXiv.2602.01692},
  url = {https://arxiv.org/abs/2602.01692}
}

@article{HZ2017,
  author = {Huang, Hao and Zhao, Yi},
  title = {Degree versions of the {E}rd{\H o}s-{K}o-{R}ado theorem and {E}rd{\H o}s hypergraph matching conjecture},
  journal = {J. Combin. Theory Ser. A},
  volume = {150},
  year = {2017},
  pages = {233--247},
  doi = {10.1016/j.jcta.2017.03.002}
}

@article{IK2026,
  author = {Ihringer, Ferdinand and Kupavskii, Andrey},
  title = {Structure of $t$-{I}ntersecting Families of Vector Spaces},
  journal = {arXiv preprint arXiv:2605.02698},
  year = {2026},
  eprint = {2605.02698},
  archiveprefix = {arXiv},
  primaryclass = {math.CO},
  doi = {10.48550/arXiv.2605.02698},
  url = {https://arxiv.org/abs/2605.02698}
}

@article{SZ2024,
  author = {Shan, Yunjing and Zhou, Junling},
  title = {$d$-Degree {E}rd{\H o}s-{K}o-{R}ado theorem for finite vector spaces},
  journal = {Acta Mathematica Hungarica},
  volume = {176},
  pages = {215--235},
  year = {2025},
  doi = {10.1007/s10474-025-01543-1},
  url = {https://arxiv.org/abs/2411.17985}
}

@article{WXZ2023,
  author = {Wang, Jun and Xu, Ao and Zhang, Huajun},
  title = {A {K}ruskal--{K}atona-type theorem for graphs: {$q$}-{K}neser graphs},
  journal = {Journal of Combinatorial Theory, Series A},
  volume = {198},
  year = {2023},
  pages = {105766},
  doi = {10.1016/j.jcta.2023.105766}
}

@article{HiltonMilner1967,
  author = {Hilton, A. J. W. and Milner, E. C.},
  title = {Some intersection theorems for systems of finite sets},
  journal = {Quarterly Journal of Mathematics},
  volume = {18},
  number = {1},
  pages = {369--384},
  year = {1967},
  doi = {10.1093/qmath/18.1.369}
}

@book{FT2018,
  author = {Frankl, P{\'e}ter and Tokushige, Norihide},
  title = {Extremal Problems for Finite Sets},
  series = {Student Mathematical Library},
  volume = {86},
  publisher = {American Mathematical Society},
  address = {Providence, RI},
  year = {2018},
  pages = {xvii+224},
  doi = {10.1090/stml/086}
}

@article{Frankl1987maxdeg,
  author = {Frankl, P{\'e}ter},
  title = {{E}rd{\H o}s-{K}o-{R}ado theorem with conditions on the maximal degree},
  journal = {J. Combin. Theory Ser. A},
  volume = {46},
  number = {2},
  year = {1987},
  pages = {252--263},
  doi = {10.1016/0097-3165(87)90005-7}
}

@article{AK1997,
  author = {Ahlswede, Rudolf and Khachatrian, Levon H.},
  title = {The complete intersection theorem for systems of finite sets},
  journal = {European J. Combin.},
  volume = {18},
  number = {2},
  year = {1997},
  pages = {125--136},
  doi = {10.1006/eujc.1995.0092}
}

@incollection{Frankl1987shifting,
  author = {Frankl, P{\'e}ter},
  title = {The shifting technique in extremal set theory},
  booktitle = {Surveys in Combinatorics 1987},
  editor = {Whitehead, C.},
  series = {London Math. Soc. Lecture Note Ser.},
  volume = {123},
  publisher = {Cambridge Univ. Press},
  address = {Cambridge},
  year = {1987},
  pages = {81--110}
}

@article{Kupavskii2018diversity,
  author = {Kupavskii, Andrey},
  title = {Diversity of uniform intersecting families},
  journal = {European J. Combin.},
  volume = {74},
  year = {2018},
  pages = {39--47},
  doi = {10.1016/j.ejc.2018.07.002}
}

@article{FW1986,
  author = {Frankl, P{\'e}ter and Wilson, Richard M.},
  title = {The {E}rd{\H o}s-{K}o-{R}ado theorem for vector spaces},
  journal = {J. Combin. Theory Ser. A},
  volume = {43},
  number = {2},
  year = {1986},
  pages = {228--236},
  doi = {10.1016/0097-3165(86)90063-4}
}

\end{document}